\documentclass[12pt]{amsart}
\usepackage{amsmath,amssymb,amsthm}
\usepackage[margin=1in]{geometry}
\usepackage{tikz}
\usetikzlibrary{matrix}
\usetikzlibrary{matrix,positioning}
\usetikzlibrary{patterns, snakes}
\usepackage{setspace}
\usepackage[normalem]{ulem}

\newtheorem{theorem}{Theorem}
\newtheorem{cor}[theorem]{Corollary}
\newtheorem{prop}[theorem]{Proposition}
\newtheorem{definition}{Definition}
\newtheorem{lemma}[theorem]{Lemma}
\newtheorem{question}[theorem]{Question}

\newcommand{\uhr}{\upharpoonright}
\newcommand{\kO}{\mathcal{O}}
\newcommand{\Nat}{\omega}
\newcommand{\eps}{\varepsilon}
\newcommand{\C}{\mathcal C}
\renewcommand{\flat}{\operatorname{Flat}}
\usepackage{soul}
\newcommand{\august}[1]{#1}
\newcommand{\mathaugust}{}

\newcommand{\strikethrough}[1]{}

\title{An effective analysis of the Denjoy rank}
\author[L.\ B.\ Westrick]{Linda Brown Westrick}
\address{Department of Mathematics\\
Penn State University\\
University Park, PA, USA}
\email{westrick@psu.edu}
\thanks{The author was partially supported by the P.E.O. International
Scholar Award.}

\begin{document}

\begin{abstract}
We analyze the descriptive complexity of several $\Pi^1_1$ ranks 
from classical analysis which are associated to Denjoy integration. 
We show that $VBG, VBG_\ast, ACG$ and $ACG_\ast$ are $\Pi^1_1$-complete, 
answering a question of Walsh in case of $ACG_\ast$.  Furthermore, 
we identify the precise descriptive complexity of the set of functions 
obtainable with at most $\alpha$ steps of the transfinite process 
of Denjoy totalization: 
if $|\cdot|$ is the $\Pi^1_1$-rank naturally associated to 
$VBG, VBG_\ast$ or $ACG_\ast$, and if $\alpha<\omega_1^{ck}$, 
then $\{F \in C(I): |F| \leq \alpha\}$ is $\Sigma^0_{2\alpha}$-complete. 
These finer results are an application of the author's previous 
work on the limsup rank on well-founded trees.  Finally,
$\{(f,F) \in M(I)\times C(I) : F\in ACG_\ast \text{ and } F'=f \text{ a.e.}\}$ 
and $\{f \in M(I) : f \text{ is Denjoy integrable}\}$ are 
$\Pi^1_1$-complete, answering more questions of Walsh.
\end{abstract}

\maketitle

Real analysis of the early 20th century featured a number of naturally 
occuring $\Pi^1_1$-complete sets.  The most prominent example 
may be the set of differentiable functions on the unit interval $I$, 
considered as a subset of $C(I)$, the metric space of continuous 
functions on $I$ with the supremum norm.

Any $\Pi^1_1$ set $A$ may be decomposed as a transfinite union 
$\cup_{\alpha < \omega_1} A_\alpha$, where each $A_\alpha$ is Borel. 
When this is done in a sufficiently uniform way, the function 
which maps each $f \in A$ to the least $\alpha$ such that 
$f \in A_\alpha$ is called a $\Pi^1_1$-rank.  Sometimes a $\Pi^1_1$ 
set has an obvious and natural $\Pi^1_1$-rank.  This is 
the case for the set of well-founded trees $T \subseteq \omega^{<\omega}$
(the usual well-founded tree rank), the collection of countable 
closed subsets of $I$ (the Cantor-Bendixson rank), and the 
collection of continuous functions obtainable by Denjoy totalization, 
as we shall see shortly.  Other times, finding a rank that could be 
considered natural is more difficult.  For example, in 
\cite{kw} 
Kechris 
and Woodin defined a suitably natural rank on the set of differentiable 
functions.

When the rank on $A$ is natural, 
it becomes meaningful to ask for the precise descriptive 
complexity of the initial segments $A_\alpha$.  This was done 
implicitly
for the well-founded tree rank in \cite{GreenbergMontalbanSlaman2013}, 
for the Cantor-Bendixson rank in {\cite{CenzerMauldin1983} and} 
\cite{Lempp}, and for the 
Kechris-Woodin rank on differentiable functions in \cite{Westrick2014}. 
The purpose of this paper is to show that the method used in 
\cite{Westrick2014} generalizes to give descriptive complexities for 
three hierarchies from classical analysis related to Denjoy 
totalization.  We also give a new proof of
a result obtained by both Cenzer and Mauldin \cite{CenzerMauldin1983} 
and independently Lempp \cite{Lempp}, concerning
the descriptive complexity of the initial segments of the 
Cantor-Bendixson rank.

Denjoy totalization is a transfinite integration process 
developed by Denjoy in 1912
to solve the problem of recovering $F$ from $F'$ whenever 
$F$ is an everywhere differentiable function in $C(I)$.  
The process does a little more, recovering also some a.e. differentiable 
$F$, but not all of them.  The set of $F \in C(I)$ 
recoverable from $F'$ by Denjoy totalization is denoted $ACG_\ast$. 
The related sets $ACG, VBG_\ast$ and $VBG$ are described in the 
next section.

The main result is the following.
Let $|\cdot|_{VB}$, $|\cdot|_{VB_\ast}$, 
$|\cdot|_{AC}$ and $|\cdot|_{AC_\ast}$ denote the natural 
$\Pi^1_1$ ranks on $VBG, VBG_\ast, ACG$ and $ACG_\ast$ 
respectively.
\begin{theorem}\label{thm:1}
Let $X = VB, VB_\ast$ or $AC_\ast$, let $Y \in 2^\omega$, 
let $1<\alpha < \omega_1^Y$, 
and let $$A_\alpha = \{F\in C(I) : |F|_X \leq \alpha\}.$$  Then
$A_\alpha$ is $\Sigma^0_{2\alpha}(Y)$, and for any 
$\Sigma^0_{2\alpha}(Y)$ set $B$, there is a $Y$-computable reduction from 
$B$ to $A_\alpha$.
In particular, $A_\alpha$ is $\mathbf \Sigma^0_{2\alpha}$-complete,
and if $\alpha< \omega_1^{CK}$, then $A_\alpha$ is $\Sigma^0_{2\alpha}$-complete.
\end{theorem}

Let $M(I)$ denote the Polish space of measurable functions on the unit 
interval, with metric given by $d(f,g) = \int_I \min(1,|f-g|)$. 
Our next theorem answers three questions in \cite{walsh}.
\begin{theorem}\label{thm:2}
The following sets are all $\Pi^1_1$-complete:
\begin{enumerate}
\item $VBG, VBG_\ast, ACG$ and $ACG_\ast$
\item $\{(f,F) \in M(I) \times C(I) : F \in ACG_\ast \text{ and } F'=f \text{ a.e.}\}$
\item $\{f \in M(I) : F \text{ is Denjoy integrable}\}$
\end{enumerate}
\end{theorem}

The sets $\{F\in C(I) : |F|_{X} \leq 1\}$ are better known 
as the collection of functions of bounded variation when $X = VB, VB_\ast$ 
and the collection of absolutely continuous functions when $X = AC, AC_\ast$. 
The following results 
filling in the $\alpha=1$ case of Theorem \ref{thm:1}
have routine proofs but their statements may be of interest.
\begin{theorem}
\begin{enumerate}
\item The set of continuous functions of bounded variation is $\Sigma^0_2$-complete.
\item The set of absolutely continuous functions is $\Pi^0_3$-complete.
\end{enumerate}
\end{theorem}

In Section \ref{sec:prelim}, after some preliminaries we 
review the main tool in \cite{Westrick2014}, an analysis of the 
descriptive complexity of the limsup rank on the set of well-founded 
trees.  We also give the 
background on Denjoy totalization and the hierarchies $ACG_\ast, 
ACG, VBG_\ast$ and $VBG$.  
In Section \ref{sec:cantorBendixson} we apply this analysis 
to the simpler case of the Cantor-Bendixson derivative, recovering the 
aforementioned result of {Cenzer, Mauldin, and } Lempp.
In Section \ref{sec:completeness} we prove hardness results for 
the hierarchies on $ACG_\ast, ACG, VBG_\ast$ and $VBG$, obtaining 
Theorem \ref{thm:2} and one direction of Theorem \ref{thm:1}.
In Section \ref{sec:descriptions} we obtain
matching descriptive results for all of these hierarchies except $ACG$, 
giving the other direction of Theorem \ref{thm:1}. 
Section \ref{sec:questions} contains open questions.

The author would like to thank Ted Slaman and Sean Walsh for a number of 
useful and interesting discussions, \august{and the anonymous referee, 
whose comments have clarified the main arguments.}

\section{Preliminaries}\label{sec:prelim}

\subsection{Notation}

We use standard computability-theoretic notation.  
Trees $T \subseteq \omega^{<\omega}$ can be encoded by elements of Cantor 
space.  Let $Tr \subseteq 2^\omega$ be the set of codes for such trees. 
Given $T \in Tr$, we usually forget the encoding and treat $T$ as 
a subset of $\omega^{<\omega}$ anyway.  
The set $[T]\subseteq \omega^\omega$ is 
the set of paths through $T$.  Let $\sigma^\smallfrown\tau$ denote the 
concatenation of $\sigma$ and $\tau$.
 If $\sigma \in T$, let $T_\sigma$ denote 
$\{\tau : \sigma^\smallfrown\tau \in T\}$, 
 and for any $\sigma \in \omega^{<\omega}$, let $\sigma^\smallfrown T$ denote 
$\{\sigma^\smallfrown \tau : \tau \in T\}$.
\august{Trees of the form $T_\sigma$ where $\sigma = \langle n \rangle$
 are commonly encountered, and for these trees we drop the brackets 
and just write $T_n$ instead of $T_{\langle n \rangle}$.}
Let $[\sigma]$ denote $\{X \in \omega^\omega\ : \sigma \prec X\}$ or 
$\{X \in 2^\omega : \sigma \prec X\}$, depending on the context. 
When a tree is in fact a subset of $2^{<\omega}$, we usually give it 
the name $S$, reserving $T$ for trees in Baire space.  

The complement of a set $A \subseteq \omega^\omega$ is denoted $\neg A$. 
If $A,B,C,D\subseteq \omega^\omega$ and $f$ is a computable function 
from $A \cup B$ to $C\cup D$ with $f^{-1}(C) = A$ and $f^{-1}(D) = B$, 
we say that $f$ is a computable reduction from $(A,B)$ to $(C,D)$.
If $B = \neg A$ and $D = \neg C$, we say $f$ is a computable reduction 
from $A$ to $C$.
 
A tree 
$T \subseteq \omega^{<\omega}$ is well-founded, denoted $T \in WF$, 
if $[T] = \emptyset$.  A set $A \subseteq \omega^\omega$ is $\Pi^1_1$ 
if there is a computable reduction from $A$ to $WF$. 
A $\Pi^1_1$ set $A$ is $\Pi^1_1$-complete 
if there is a computable function from $WF$ to $A$.
The boldface notions $\mathbf \Pi^1_1$ and $\mathbf \Pi^1_1$-complete 
are obtained by replacing ``computable'' with ``continuous'', 
or equivalently, relativizing to an oracle.

For $n<\omega$, a set $X \in 2^\omega$ 
is $\Sigma^0_n$-complete if $X \equiv_1 \emptyset^{(n)}$.  To extend 
the notion of completeness to the ordinals,
we assume some familiarity with hyperarithmetic theory: Kleene's $\kO$,
ordinal notations, the sets $H_a$ for $a \in \kO$.  For $a \in \kO$, 
we use $|a|_\kO$ to denote the ordinal coded by $a$.  The supremum of 
ordinals represented in $\kO$ is $\omega_1^{CK}$.
Following \cite{ak}, if $n\geq \omega$
we say that a subset $X \subseteq \omega$ is 
$\Sigma^0_\alpha$-complete if
$X \equiv_1 H_{2^a}$, where $a \in \kO$ is any notation with 
$|a|_\kO = \alpha$. All these notions can relativize to an oracle 
$Y \in 2^\omega$, using notation $\kO^Y, |a|_\kO^Y, H_a^Y, \omega_1^Y, 
\Sigma^0_{\alpha}(Y)$. 

These computability notions on subsets of $\omega$ can be extended to 
effective topological notions on subsets of $\omega^\omega$.
A set $A \subseteq \omega^\omega$ is $\Sigma^0_{\alpha}(Y)$ if 
$\omega \leq \alpha < \omega_1^Y$ and there is an index 
$e_0$ and a notation $a \in \kO^{Y\oplus \emptyset}$ 
such that for all $X \in \omega^\omega$, 
$a \in \kO^{Y\oplus X}$ with $|a|^{Y\oplus X} = \alpha$ and
$$X \in A \iff e_0 \in H_{2^a}^{Y\oplus X}.$$
If $\alpha < \omega$, replace $H_{2^a}^{Y\oplus X}$ with $H_a^{Y\oplus X}$. 
When showing that a set is $\Sigma^0_{\alpha}(Y)$, we usually use 
effective transfinite recursion.  We assume familiarity with this 
method and sometimes avoid explicit mention of the sets 
$H_{2^a}^{Y\oplus X}$ while using it.

It is well-known that 
$\mathbf \Sigma^0_\alpha = \bigcup_{\substack{Y \in 2^\omega\\ \alpha < \omega_1^Y}} \Sigma^0_\alpha(Y)$,
where $\mathbf \Sigma^0_\alpha$ 
refers to the $\alpha$th level of the Borel hierarchy.

We use $I$ to refer to the unit interval, $C(I)$ for the space 
of continuous real-valued functions on $I$ with the supremum norm, and 
$M(I)$ for the space of measurable functions on the unit 
interval, with metric given by $d(f,g) = \int_I \min(1,|f-g|)$. 
In $C(I)$, the piecewise linear functions with rational endpoints 
form a countable dense subset; computing a element $F$ of $C(I)$ means 
providing, uniformly in $\eps$, a piecewise linear function within $\eps$ 
of $F$ in the supremum norm.
In $M(I)$, the step functions with 
rational values and discontinuities at 
finitely many rational points form a countable 
dense subset, and can be used as the ideal points in a representation 
of $M(I)$ as an effectively presented metric space.  Thus
computing an element $f$ of $M(I)$ means providing, uniformly in $\eps$, 
an ideal point of the type described above which is $\eps$-close to $f$. 
\august{A closed subset $P$ of $I$ is always coded by an enumeration of the 
open intervals $(p,q)$ with rational endpoints such that $[p,q] \cap P = \emptyset$.}

\subsection{Limsup rank}
Our main tool is the \emph{limsup rank} 
on well-founded trees, which was defined in \cite{Westrick2014} 
as follows.

\begin{definition}
Let $T\subseteq \omega^{<\omega}$ be a well-founded tree.  If $T = \emptyset$, 
define $|T|_{ls} = 0$.  Otherwise, define
$$|T|_{ls} = \max( \sup_n |T_n|_{ls}, (\limsup_n |T_n|_{ls}) + 1 ).$$
\end{definition}

This rank is designed to line up nicely with Cantor-Bendixson type 
derivation processes in a way that will be explained below.

\begin{theorem}[\cite{Westrick2014}]\label{ladrthm}
For all constructive $\alpha > 1$,
$\{e : \phi_e \text{ codes a tree $T$ with } |T|_{ls} \leq \alpha\}$ is 
$\Sigma_{2\alpha}$-complete.
\end{theorem}

The only reason for not allowing $\alpha = 1$ in the above theorem 
is that it is $\Pi_2$ to tell whether $\phi_e$ is total, that is, 
whether it codes anything at all.  

Almost for free, we can make topological claims in addition to 
computational ones.  The next theorem follows from Theorem 
\ref{ladrthm} by relativization.

\begin{theorem}\label{ladredst}
For any nonzero $\alpha < \omega_1$, let  
$$A_\alpha = \{T \in Tr : |T|_{ls} \leq \alpha\}.$$  Then
if $\alpha < \omega_1^Y$, we have $A_\alpha \in \Sigma^0_{2\alpha}(Y)$, 
and for any set $B \in \Sigma^0_{2\alpha}(Y)$, there is a $Y$-computable 
reduction from $B$ to $A_\alpha$.  In particular, $A_\alpha$ is 
$\mathbf \Sigma^0_{2\alpha}$-complete, and if $\alpha < \omega_1^{CK}$, 
then $A_\alpha$ is $\Sigma^0_{2\alpha}$-complete.
\end{theorem}
\begin{proof} Let $a$ be a notation such that for all $X$, 
$a \in \kO^{X\oplus Y}$ and $|a|_{\kO}^{X\oplus Y} = 2\alpha$.
For all $X$, by relativization we have
\begin{equation}\label{eqn}
H_{2^a}^{X\oplus Y} \equiv_1 \{e : \phi_e^{X\oplus Y} 
\text{ codes a tree $T$ with } |T|_{ls} \leq \alpha\},
\end{equation}
where the pair of reductions witnessing the 1-equivalence 
do not depend on $X$.\footnote{Because both sets consist of machine 
indices for machines with access to $X\oplus Y$, the existence 
of a single pair of computable reductions follows 
from the existence of a single pair of uniformly $X\oplus Y$-computable 
reductions.}  Letting $d_0$ be such that $\phi_{d_0}^{X\oplus Y} = X$,
the reverse reduction in (\ref{eqn}) provides $e_0$ such that for 
all $X$,
$$X \in A_\alpha \iff \phi_{d_0}^{X\oplus Y} \in A_\alpha 
\iff e_0 \in H_{2^a}^{X\oplus Y},$$
so $A_\alpha \in \Sigma^0_{2\alpha}(Y)$.

To show that a given $B \in \Sigma^0_{2\alpha}(Y)$ can be 
$Y$-computably reduced to $A_\alpha$, it suffices to consider
$B = \{X : n_0 \in H_{2^a}^{X\oplus Y}\}$ where $n_0$ is chosen 
to make $B$ $\Sigma^0_{2\alpha}(Y)$-universal.  The forward 
reduction in (\ref{eqn}) provides $m_0$ such that 
$X \in B \iff \phi_{m_0}^{X\oplus Y} \in A_\alpha$, and the mapping 
$X\mapsto \phi_{m_0}^{X\oplus Y}$ is $Y$-computable.
\end{proof}

\subsection{Denjoy totalization}
Classical real analysis includes the study of absolutely continuous 
functions, functions of bounded variation, and countable generalizations 
of these notions.  We consider four classes of real-valued 
functions on $I$.  They are $VBG$ (generalized bounded variation), 
$VBG_\ast$ (generalized bounded variation in the restricted sense), 
$ACG$ (generalized absolutely continuous) and $ACG_\ast$ (generalized 
absolutely continuous in the restricted sense).

To define these classes it is necessary to generalize the well-known 
definitions of bounded variation and absolute continuity to take into 
account also a closed set $E$ to which the function $F$ should be in some 
sense restricted.  The non-asterisk definitions are the literal 
restrictions.  The others take into account also the values 
of $F$ outside of $E$.  The \emph{oscillation} of a function 
$F$ on an interval $(a,b)$ is defined as 
$\omega(F,a,b) = \sup_{x,y \in (a,b)}
|F(x) - F(y)|$.

\begin{definition} Let $F\in C(I)$ and let $E\subseteq I$.
\begin{enumerate}
\item We say $F$ is $VB$ (respectively $VB_\ast$) on $E$
if there is an $N$ such that for all non-decreasing 
sequences $a_0,b_0,\dots a_k,b_k \in E$, 
we have $\sum_i |F(b_i) - F(a_i)| < N$ 
(respectively $\sum_i \omega(F,a_i,b_i) < N$). 
\item We say $F$ is $AC$ (respectively $AC_\ast$) on $E$
if for all $\eps$ there is a $\delta$ such 
that for all non-decreasing sequences $a_0,b_0,\dots a_k,b_k \in E$, if 
$\sum_i |b_i-a_i| < \delta$, then $\sum_i |F(b_i) - F(a_i)|<\eps$ 
(respectively $\sum_i \omega(F,a_i,b_i) < \eps$).
\end{enumerate}
\end{definition}

Observe that if $E$ is an interval, then being $VB$ on $E$ is the 
same thing as being $VB_\ast$ on $E$, and similarly for absolute 
continuity.  We can also understand what it means for 
a function to satisfy these properties on a closed set $E$ 
with reference to simplified 
functions $F_E$ and $F_{E,\ast}$ defined as follows.

\begin{definition}
Let $F\in C(I)$, and $E\subseteq I$ a closed set.  Then let $F_{E}$ 
and $F_{E,\ast}$ denote the functions satisfying
\begin{enumerate}
\item $F_{E}(x) = F_{E,\ast}(x) = F(x)$ for $x\in E$, and 
\item If $(c,d)$ is a connected component of $I\setminus E$, 
\begin{enumerate}
\item let $F_E$ be linear on $[c,d]$, and 
\item let 
\begin{align*}F_{E,\ast}\left(\frac{2c+d}{3}\right) = \sup F([c,d]) \\
F_{E,\ast}\left(\frac{c+2d}{3}\right) = \inf F([c,d]),\end{align*} 
and let $F_{E,\ast}$ be linear 
on $[c, \frac{2c+d}{3}], [\frac{2c+d}{3}, \frac{c + 2d}{3}]$ and 
$[\frac{c+2d}{3},d]$.
\end{enumerate}
\end{enumerate}
\end{definition}
Note that $\omega(F_{E,\ast}, [c,d]) = \omega(F,[c,d])$ 
for $(c,d)$ a connected component of $I\setminus E$.

Then the following proposition holds:
\begin{prop}\label{prop:linearized}
 Let $F \in C(I)$ and $E\subseteq I$ be closed.  Then
\begin{enumerate}
\item $F$ is $VB$ (resp. $AC$) on $E$ if and only if $F_E$ is $VB$ 
(resp. $AC$) on $I$.
\item $F$ is $VB_\ast$ (resp. $AC_\ast$) on $E$ if and only if 
$F_{E,\ast}$ is $VB$ (resp. $AC$) on $I$.
\end{enumerate}
\end{prop}

Now we can define the main notions.

\begin{definition}
A function $F\in C(I)$ is $VBG$ (respectively $ACG, VBG_\ast, ACG_\ast$) 
if there is a 
countable sequence of closed sets $E_n$ such that $\cup_n E_n = I$ 
and $F$ is $VB$ (respectively $AC, VB_\ast, AC_\ast$) on each $E_n$.
\end{definition}

It is immediate that $VBG_\ast \subseteq VBG$ and $ACG_\ast \subseteq ACG$.
Recall the relationship between absolute continuity and bounded variation:
a continuous function of bounded variation is absolutely continuous if and
only if it satisfies Lusin's condition $(N)$.

\begin{definition}
A function $F:I\rightarrow \mathbb R$ satisfies $(N)$ if for every Lebesgue 
null set $A\subseteq I$, its image $F(A)$ is also null.
\end{definition}

If $\cup_n E_n = I$, then $F$ satisfies $(N)$ if and only if it 
satisfies $(N)$ on each $E_n$.  Therefore, $F \in ACG$ if and only if 
$F \in VBG$ and $F$ satisfies $(N)$.  We also have (see \cite[Thm VII.8.8, 
pg. 233]{saks}) that $ACG_\ast = VBG_\ast \cap ACG$.  It follows that 
$F \in ACG_\ast$ if and only if $F \in VBG_\ast$ and $F$ satisfies $(N)$.

Based on the definitions above, it would seem that these sets are 
$\Sigma^1_2$.  However, the following equivalent characterization
shows each of these classes is in fact $\Pi^1_1$.

\begin{theorem}[see {\cite[Thm VII.9.1, pg 233]{saks}}]\label{thm:7}
For a function $F$ to be $VBG$ (respectively $VBG_\ast,ACG, ACG_\ast$) 
it is necessary and sufficient that for every closed $E \subseteq I$, 
there is an interval $[a,b]\subseteq I$ such that 
$(a,b) \cap E \neq \emptyset$ 
and $F$ is $VB$ 
(respectively $VB_\ast, AC, AC_\ast$) on $[a,b] \cap E$.
\end{theorem}

\begin{cor}
The sets $VBG, VBG_\ast, ACG, ACG_\ast \subseteq C(I)$ are all 
$\Pi^1_1$.
\end{cor}

The previous theorem also suggests a derivation process.
\begin{definition}
Let $X$ stand for $VB, VB_\ast, AC$ 
or $AC_\ast$.  Given $F \in C(I)$, define $P^0_{F,X} = I$. 
Define 
$$P^{\alpha + 1}_{F,X} = P^\alpha_{F,X} \setminus \cup \{(a,b) : F \text{ is 
$X$ on } [a,b]\cap P^\alpha_{F,X}\}$$
and for a limit ordinal $\lambda$, define 
$P^\lambda_{F,X} = \cap_{\alpha< \lambda} P^\alpha_{F,X}$.
Define a rank $|\cdot|_X$ by letting $|F|_X$ be the least 
$\alpha$ such that $P^\alpha_{F,X} = \emptyset$, if such $\alpha$ 
exists.
\end{definition}

If $F \in VBG$, then the only way for $P^{\alpha+1}_{F,VB} = P^\alpha_{F,VB}$ is 
if $P=\emptyset$; and furthermore, the countable 
sequence of sets $[a,b]\cap P^\alpha_{F,VB}$ for which $(a,b)$ 
were removed over the course of the derivations would serve as 
the sequence $E_n$ required in the definition of $VBG$.
Reasoning similarly about all four hierarchies,
it is immediate that a given $F \in C(I)$ belongs to one of the 
classes $VBG, VBG_\ast, ACG, ACG_\ast$ if and only if the associated 
derivation process eventually produces the empty set.  

\begin{definition}
Let $VBG_\alpha$ (respectively $VBG_{\ast\alpha}, ACG_\alpha, ACG_{\ast\alpha}$) 
denote the sets $\{F \in C(I): |F|_{VB} \leq \alpha\}$ (respectively 
$|F|_{VB_\ast}, |F|_{AC}, |F|_{AC_\ast}$).
\end{definition}

Recall from the introduction that $ACG_\ast$ is exactly the set of 
functions $F \in C(I)$ which can be recovered from $F'$ by 
Denjoy totalization. 
We will not give the
definition of the Denjoy totalization process, 
\august{because we almost always prefer to work with
the related derivation 
processes. The one place where familiarity with 
the totalization process is needed
is in the proof of Theorem} \ref{ajtai}.  \august{A definition 
can be found in} \cite[Section VIII.5]{saks}. 
What matters for us is that Denjoy totalization is a transfinite 
procedure which terminates at some countable ordinal stage, and 
$ACG_{\ast\alpha}$ consists of precisely the functions $F \in ACG_\ast$ 
which are recovered from $F'$ in at most $\alpha$ steps of Denjoy 
totalization.  Therefore, the sets $ACG_{\ast\alpha}$ have a meaningful 
interpretation in terms of Denjoy totalization.

Note that there are actually two transfinite procedures which 
are sometimes called ``Denjoy totalization'': the narrow Denjoy 
integral, which coincides with the integrals of Perron, Kurzweil 
and Henstock, and the wide Denjoy integral, sometimes known as 
the Denjoy-Khintchine integral.  In this paper, ``Denjoy totalization'' 
always refers to the narrow Denjoy integral.  The narrow Denjoy 
integral has the same relationship to the class $ACG_\ast$ as 
the wide Denjoy integral has to the class $ACG$.

\section{Cantor-Bendixson rank}\label{sec:cantorBendixson}

In this section we analyze the initial segments of the Cantor-Bendixson 
hierarchy.  The theorems of this section are not used in later sections.
Let $S \subseteq 2^{<\omega}$ be a tree with no dead ends.  
Let $[S]$ denote the set of paths in $S$. 
The Cantor-Bendixson derivative $D(S)$ is defined as the tree without dead 
ends  such that $[D(S)]$ consists of exactly the paths not isolated in $[S]$.  
Define $D^0(S) = S$, $D^{\alpha+1}(S) = D(D^\alpha(S))$, and 
$D^\lambda(S) = \cap_{\alpha<\lambda} D^\alpha(S)$ for $\lambda$ a limit.
\begin{definition}
The Cantor-Bendixson rank of a tree $S\subseteq 2^\omega$, denoted $|S|_{CB}$, is the least $\alpha$ such that $D^\alpha(S) = \emptyset$, if such exists.  Otherwise we say $|S|_{CB} = \infty$.
\end{definition}
Other authors define $|T|_{CB}$ to be the least $\alpha$ such that 
$D^\alpha(S) = D^{\alpha+1}(S)$, so that a set with a perfect subset also has a 
Cantor-Bendixson rank.  Others define their rank to be always one less than ours, 
so that every ordinal is used.

\begin{prop}
Let $Y \in 2^\omega$ and $\alpha < \omega_1^Y$.  Then $\{S : |S|_{CB} \leq \alpha\}$ 
is $\Sigma^0_{2\alpha}(Y)$ in $\{S : S \text{ is a tree with no dead ends}\}$.
\end{prop}

\begin{proof}
The proof is by effective transfinite recursion.  
Because checking whether a no-dead-ends
tree is empty can be accomplished by checking the root, 
$\{ S : |S|_{CB} = 0\}$ is $\Sigma^0_0$.

A tree $S$ has $D^{\alpha+1}(S) = \emptyset$ if and only if 
$D^\alpha(S)$ has only finitely many branches.
If $D^\alpha(T)$ has at least $k$ branches, then by going up to a height $n$ at which the branches have separated, we may find at least $k$-many $\sigma$ of length $n$ such that $D^\alpha(T_\sigma) \neq \emptyset$.  And if there are $k$ incomparable $\sigma$ such that $D^\alpha(T_\sigma)\neq \emptyset$, then $D^\alpha(T)$ has at least $k$ branches.  Thus
$$D^{\alpha+1}(S) = \emptyset \iff \exists k \forall n (\text{there are at most $k$ many $\sigma$ of length $n$ for which } D^\alpha(S_\sigma) \neq \emptyset).$$    
Assuming $\{S : D^\alpha(S)=\emptyset\}$ is $\Sigma^0_{2\alpha}(Y)$ uniformly in $\alpha$, 
this shows that $D^{\alpha+1}(S)=\emptyset$ is $\Sigma^0_{2\alpha+2}(Y)$ uniformly
in $\alpha$.

If $\lambda$ is a limit, a tree has $D^{\lambda}(S) = \emptyset$ if and only if there is an $\alpha<\lambda$ such that $D^\alpha(S) = \emptyset$, by compactness.  Assuming $D^{\alpha}(S) = \emptyset$ is $\Sigma^0_{2\alpha}(Y)$ 
uniformly in $\alpha < \lambda$, and supposing a sequence
$\alpha_n\rightarrow \lambda$ is $Y$-effectively given, 
we have $D^\lambda(S) = \emptyset$ if and only if $\exists \alpha < \lambda[D^\alpha(S)=\emptyset]$, a $\Sigma^0_\lambda(Y)$ statement.  Note $\Sigma^0_\lambda(Y) =\Sigma^0_{2\lambda}(Y)$ for $\lambda$ a limit.
\end{proof}

\begin{prop}
There is a computable reduction $T\mapsto S_T$ from trees 
$T \subseteq \omega^{<\omega}$ to no-dead-end trees $S\subseteq 2^{<\omega}$, 
satisfying
\begin{enumerate}
\item If $T$ is not well-founded, $[S_T]$ contains a perfect set.
\item If $T$ is well-founded with $|T|_{ls} = \alpha$, then $[S_T]$ is countable, 
and $|S_T|_{CB} = \alpha$.
\end{enumerate}
\end{prop}
\begin{proof}
The idea is that each node of a tree $T \subseteq \Nat^{<\Nat}$ 
should correspond to a branching of paths in $S_T\subseteq 2^{<\omega}$, 
with the topological clustering of the paths provided by the hierarchical structure of $T$.

Define $S_T$ by
$$S_T = \{0^{2n_0+i_0}10^{2n_1+i_1}1\cdots0^{2n_k+i_k}10^m : 
(n_0,\dots,n_k)\in T, i_0,\dots i_k \in \{0,1\}\}.
\footnote{The more familiar option, 
$S_T' = \{0^{n_0}10^{n_1}1\cdots0^{n_k}10^m : (n_0,\dots,n_k)\in T\}$, satisfies 
part (2) but not part (1) of the theorem; consider the case when $[T]$ consists 
of a single path.}$$
Note that $0^m \in S_T$ if and only if the empty node is in $T$.
If $T\not\in WF$, then if $P \in [T]$ the following subset of 
$[S_T]$ is perfect:
$$\{0^{2P(0) + X(0)}10^{2P(1) + X(1)}1\dots : X \in 2^\omega\}.$$

Supposing now that $T$ is well-founded, 
we claim that $|T|_{ls} = |S_T|_{CB}$.  
The proof is by induction on the usual rank of $T$.
If $T = \emptyset$ then also $S_T = \emptyset$, so $|S_T|_{CB} =  |T|_{ls} = 0$.

Suppose $|T|_{ls} = \alpha+1$.  (This is the only case because the limsup rank is 
always a successor.)  Then there is an $N$ such that for $n\geq N$, 
$|T_n|_{ls} \leq \alpha$.  
Observe that if $C\subseteq 2^{<\omega}$ is any finite prefix-free 
collection of strings such that 
$\cup_{\sigma\in C} [\sigma]$ covers $[S]$, 
then $$D^\alpha(S) = \bigcup_{\sigma\in C} \sigma^\smallfrown D^\alpha(S_\sigma).$$ 
Letting 
$C = \{0^{2N}\} \cup \{0^{2n+i}1 : n < N, i\in \{0,1\}\}$,
we have 
$$D^{\alpha+1}(S_T) = \bigcup_{\sigma\in C} 
\sigma^\smallfrown D^{\alpha+1}((S_T)_\sigma).$$
By induction, for $n< N$ and $i \in \{0,1\}$,  
$|0^{2n+i}1^\smallfrown S_{T_n}|_{CB}\leq \alpha + 1$, so
$$D^{\alpha+1}(S_T) = {0^{2N}}^\smallfrown D^{\alpha+1}((S_T)_{0^{2N}}).$$
Also, for $n \geq N$, we have 
$|0^{2n+i}1^\smallfrown S_{T_n}|_{CB} \leq \alpha$, 
so for each $k$, $D^\alpha((S_T)_{0^{2N+k}1}) = \emptyset$, so 
$[D^\alpha((S_T)_{0^{2N}})] \subseteq \{0^\omega\}$.  
Therefore $D^{\alpha+1}((S_T)_{0^{2N}}) = \emptyset$, so 
$|S_T|_{CB} \leq \alpha + 1$.

Now we need $|S_T|_{CB} \geq \alpha + 1$. If $|T_n|_{ls} = \alpha + 1$ for some $n$, then by induction $|0^n1^\smallfrown S_{T_n}|_{CB}=\alpha+1$, which suffices.
If $\limsup_n |T_n|_{ls} = \alpha$, then for every $\beta < \alpha$, there are infinitely many $n$ such that $|T_n| > \beta$, so there are infinitely many $n$ for which $|0^n1^\smallfrown S_{T_n}|_{CB} > \beta$.  Therefore, for all $\beta<\alpha$, $0^\omega$ is not an isolated
path of $D^\beta(S_T)$, so $0^\omega \in [D^\alpha(S_T)]$, and $|S_T|_{CB} > \alpha$.
 \end{proof}

{We can now obtain a new proof of the results 
of Cenzer and Mauldin in \cite{CenzerMauldin1983} and 
\cite{CenzerMauldin1982}.  Though their results are stated 
topologically, their reductions are computable.}

\begin{theorem}[{\cite[Theorem 8]{CenzerMauldin1983}}]
For any nonzero $\alpha<\omega_1$, let 
$$A_\alpha = \{S \subseteq 2^\omega : S \text{ is a no dead end tree and }
|S|_{CB}\leq\alpha\}.$$  Then all the conclusions of Theorem \ref{ladredst} 
hold.  In particular, if $\alpha<\omega_1^{CK}$, 
$A_\alpha$ is $\Sigma^0_{2\alpha}$-complete.
\end{theorem}

{Their main result, identifying the precise Borel class of the 
$\alpha$-derivative operator 
on the space of closed subsets of $2^\omega$, follows 
directly from this theorem, 
and they also prove this theorem on the way to their result. 

A similar index set result was obtained independently by Lempp.}

 \begin{cor}{\cite{Lempp}}
 For each constructive $\alpha > 1$, the sets $$\{ e : \phi_e \text{ codes a tree $S$ which has no dead ends and } |S|_{CB}\leq \alpha\}$$  are 
$\Sigma_{2\alpha}$-complete.
 \end{cor}

This result differs only cosmetically from the result as it was stated in
in \cite{Lempp} because 
he considered trees which might have dead ends and because he used 
a definition of $\Sigma^0_{2\alpha}$-complete which is off by one from 
our definition for $\alpha \geq \omega$.  
Due to a similar difference in notational conventions, the results of 
\cite{CenzerMauldin1983} also appear to be off by one from ours in 
the infinite case.

\section{Completeness and hardness on the Denjoy hierarchies}\label{sec:completeness}

In this section we present a construction which provides a 
computable reduction from $(WF, \neg WF)$ to $(ACG_\ast, \neg VBG)$.  
Because $ACG_\ast \subseteq ACG \cap VBG_\ast$ and 
$ACG \cup VBG_\ast \subseteq VBG$, 
this shows that all four classes are $\Pi^1_1$-complete.  Additionally, 
this reduction serves as a simultaneous uniform reduction from 
$(A_\alpha, A_{\alpha + 1})$ to $(ACG_{\ast\alpha}, \neg VBG_\alpha)$,
where $A_\alpha = \{T \in WF : |T|_{ls} \leq\alpha\}$.
We have $ACG_{\ast\alpha} \subseteq ACG_\alpha \cap VBG_{\ast\alpha}$ and 
$ACG_\alpha \cup VBG_{\ast\alpha} \subseteq VBG_\alpha$, so all these 
sets are at least as complex as $A_\alpha$, namely at least 
$\Sigma^0_{2\alpha}(Y)$-hard.

The idea is that each node of the given tree should contribute a finite length to the variation of the constructed function.  In most cases the total variation will be infinite as a result, but the way in which that infinite length is distributed will determine the rank of the function.

The following notation will be useful: if $J,K\subseteq I$ are two intervals, 
then $J\langle K\rangle$ 
denotes the interval that has the same relation to $K$ as 
$J$ has to $I$; that is, if $J = [a,b]$ and $K = [c,d]$ then 
$J\langle K\rangle = [c+a(d-c), c+b(d-c)]$.

To define this reduction, it will be useful to have a way to tell $F_T$ to 
increase its variation in a given interval $J$ 
in response to seeing more of $T$. 
Given an interval $J \subseteq [0,1]$, we define the wiggle function $W(J)$ 
as follows.  The goal is to have a function whose variation is at least 1 
(regardless of how small $J$ is), but which only takes values in $[0,|J|]$, 
where $|J|$ is the length of the interval $J$.  Therefore, for small $J$, 
the function should oscillate intensely.  We also leave some space at 
the top and bottom of each oscillation to give room for adding some more 
oscillations; this sets us up for a typical method of producing a function of 
high Denjoy rank.  
\begin{definition} Given an interval $J\subseteq I$, 
let $M$ be the least integer large enough that $2M|J| > 1$, and define
\begin{enumerate}
\item $W(J)(x) = 0$ if $x \not\in J$.
\item $W(J)(x) = 0$ if 
$x \in [\frac{4k}{4M+1}, \frac{4k+1}{4M+1}]\langle J\rangle$ for $k\leq M$.
\item $W(J)(x) = |J|$ if
$x \in [\frac{4k+2}{4M+1}, \frac{4k+3}{4M+1}]\langle J\rangle$ for $k<M$.
\item If $x \in [\frac{4k+i}{4M+1}, \frac{4k+i+1}{4M+1}]\langle J\rangle$ 
for $k<M$ and $i=1,3$, determine the value of $W(J)(x)$ by linear 
interpolation from the values already defined.
\end{enumerate}
Define 
\begin{align*}\flat(J) &= \{K \subseteq J : K \text{ is a maximal interval on which }
W(J) \text{ is constant}\}\\
 &= \left\{\left[\frac{4k+i}{4M+1}, \frac{4k+i+1}{4M+1}\right]\langle J\rangle : \mathaugust{(k < M, i = 0,2) \text{ or } (k = M, i = 0)}\right\}\end{align*}
\end{definition}

Fix an infinite sequence of disjoint closed interval subsets of $I$, 
decreasing in size as they 
approach 0, with gaps between them.  For specificity, say
$J_n = [\frac{1}{3n+2},\frac{1}{3n+1}]$.
For any interval $L$, we define the function $\tilde F(T,L)$ by 
recursion on the usual rank of $T$ as follows.
\begin{definition}\label{tildeF} Given an interval $L\subseteq I$, let $\tilde F(\emptyset, L)$ 
be the constant function 0, and for nonempty $T$, let
$$\tilde F(T,L) = W(L) +
\sum_{K \in \flat(L), n \in \omega} \tilde F(T_n, J_n\langle K\rangle).$$
\end{definition}
Observe that the $\tilde F(T_n,J)$ are being copied onto the plateaus of the wiggle 
functions. \august{Also, for $K,H \in \flat(L)$, the functions 
$\tilde F(T_n, J_n\langle K \rangle)$ and 
$\tilde F(T_m, J_m\langle H \rangle)$ have disjoint support,
unless $m=n$ and $K=H$.  So for $x \not\in \cup_{n,K} J_n\langle K \rangle$, 
we have $\tilde F(T,L)(x) = W(L)(x)$, and otherwise 
$\tilde F(T,L)(x) = W(L)(x) + \tilde F(T_n, J_n\langle K \rangle)(x)$ 
for the unique $n$ and $K$ such that $x \in J_n\langle K \rangle$. 
Furthermore, there is an open neighborhood of 
$J_n\langle K \rangle$ on which the functions 
$\tilde F(T,L)$ and $\tilde F(T_n,J_n\langle K \rangle)$ differ by 
only a constant.  Therefore,}
$$\mathaugust{P^\alpha_{\tilde F(T,L),X} \cap J_n\langle K \rangle = 
P^\alpha_{\tilde F(T_n,J_n\langle K \rangle),X} \cap J_n\langle K \rangle}$$
\august{for all $\alpha$, where $X$ is any of $VB, VB_\ast, 
AC, AC_\ast$.  This fact will be used throughout.}

It is possible to extend this recursive definition of $\tilde F$ so that 
any tree $T$ (even ill-founded) can be used.  Each $\sigma \in T$ 
contributes some wiggle to $\tilde F(T,L)$, in locations which can be described 
as follows.  Let $\C_\emptyset^L = \{L\}$,  
and given $\C_\sigma^L$\strikethrough{for $|\sigma| \geq 1$}, define
$$\C_{\sigma^\smallfrown n}^L = \bigcup_{H \in \C_\sigma^L} \{J_n\langle K\rangle : K \in \flat(H)\}.$$

\begin{definition} Given an interval $L\subseteq I$, define
$$F(T,L) = \sum_{\substack{\sigma \in T \\ H \in \C_\sigma^L}} W(H).$$
\end{definition}

\begin{prop} For all 
trees $T\subseteq \omega^{<\omega}$,
 $F(T,L)$ is well-defined and 
uniformly $T$-computable.
For $T \in WF$, $F(T,L) = \tilde F(T,L)$.
\end{prop}
\begin{proof}
We may decompose $F(T,L) = \sum_{\ell < \omega} F_\ell(T,L)$, where
$$F_\ell(T,L) = \sum_{\substack{\sigma \in T : |\sigma| = \ell \\ H \in \C_\sigma^L}} W(H),$$
Observe that since each $J_n$ satisfies $|J_n|<1/2$, each $H \in \C_{\sigma}^L$ 
satisfies \august{$|H| \leq 2^{-|\sigma|}$}.  Furthermore, the intervals of 
$\bigcup_{\sigma : |\sigma| = \ell} \C_\sigma^L$ 
are disjoint.  Therefore, since $W(H)$ has its range in $[0,|H|]$, 
each $F_\ell(T,L)$ has its range in $[0,2^{-\ell}]$.  Also, the $F_\ell(T,L)$ 
are uniformly computable.  Therefore, $F(T,L)$ is the effective uniform 
limit of computable functions, and is therefore computable.  

Now suppose $T \in WF$.  If $T$ is just a root,
it is clear that $F(T,L)= \tilde F(T,L)$.  Assuming that each 
$F(T_n,K) = \sum_{\sigma \in T_n, H \in \C_\sigma^K} W(H)$, the agreement 
of the two definitions follows in general because for each $n$, 
$$\C_{n^\smallfrown \sigma}^L = \bigcup_{K \in \flat(L)} \C_\sigma^{J_n\langle K\rangle}.$$
\end{proof}

\begin{prop}\label{prop:14} 
If $T$ is not well-founded, then $F(T,I) \not\in VBG$.
\end{prop}
\begin{proof}
Let $Z$ be a path in $T$, and consider 
$$P_Z = \bigcap_{\ell\in \omega} \left(\bigcup \C_{Z\uhr \ell}^I\right)$$
This is an intersection of a decreasing sequence of closed sets in a 
compact space, so $P_Z$ is closed and non-empty.  
We claim that there is no open interval $U$ such that $P_Z \cap U$ is 
nonempty and $F(T,I)$ has 
bounded variation on $P_Z \cap U$.  By Theorem \ref{thm:7}, this suffices.

Fix an open interval $U$ and a number $N$.  Since $P_Z \cap U \neq \emptyset$, 
there is an $M$ and $H_0 \in \C_{Z\uhr M}^I$ such that $P_Z\cap H_0 \subseteq U$.  
We will show that the variation of 
$F(T,I)$ on $P_Z \cap H_0$ is at least $2^N$.
Let $\C = \{H \in \C_{Z\uhr (M+N)}^I : H \subseteq H_0\}$.  Since each interval of 
$\C_{Z\uhr \ell}^I$ is broken into multiple intervals 
in $\C_{Z\uhr (\ell+1)}^I$, there are at least $2^N$ intervals in $\C$.

Let the functions $F_\ell$ be as in the previous proposition.
For each $H \in \C$, $\sum_{\ell < M+N} F_\ell(T,I)$ is constant on $H$, 
the function $W(H)$ has variation at least 1, 
and $F_{M+N}(T,I)(x)=W(H)(x)$ for all $x \in H$.
Therefore, the function $G$ defined by 
$G=\sum_{\ell \leq M+N} F_{\ell}(T,I)$ has variation at least 1 on $H$.  
In fact, $G$ has variation at least 1 on any subset of $H$ 
which contains at least one point of $K$ for each $K \in \flat(H)$.  
Therefore, it suffices to show that each such $K$ 
contains a point $x \in P_Z$ such that $G(x) = F(T,I)(x)$. 

We claim that $x= \min (P_Z \cap K)$ is such a point.
For each $\ell > M+N$, there is a unique interval 
$J \in \cup_{\sigma \in T : |\sigma| = \ell} 
\mathcal C^I_{\sigma}$ such that $x \in J$.  All those intervals are disjoint, 
so $F_\ell(T,I)(x) = W(J)(x)$.  We claim that $W(J)(x) = 0$.
Since $x \in P_Z$, $J \in \mathcal C^I_{Z \uhr \ell}$.  
Because $x$ is the minimum 
of $P_Z \cap K$, it is also the minimum of $P_Z \cap J$. 
But $P_Z$ contains points from every $L \in \flat(J)$.  
Therefore, $x$ is an element of the first 
such interval $L \subseteq J$.  
Since $W(J) \equiv 0$ on its first interval,
$W(J)(x) = 0$, as required.  Since this is true for arbitrary $\ell$, we have 
$G(x) = F(T,I)(x)$ for all such $x$.  Therefore, the variation of $F(T,I)$ on 
$H_0$ is at least $2^N$, so $F(T,I) \not\in VBG$.
\end{proof}

\begin{prop} If $T$ is well-founded and 
$|T|_{ls}\geq 1$, then $F(T,I) \in ACG_\ast$, and 
$|F(T,I)|_X = |T|_{ls}$, where $X$ is any of $VB, VB_\ast, AC, AC_\ast$.
\end{prop}
\begin{proof}
We proceed by induction on the usual rank of the tree, starting with the 
root-only tree 
$T=\{\emptyset\}$.  The statement we prove inductively is slightly stronger: 
for all intervals $L$, 
if $|T|_{ls} = \alpha+1$, then 

\begin{itemize}
\item $|F(T,L)|_X = \alpha+1$, and 
\item \august{for each $K \in \flat(L)$, there is a point $x_K \in K$ such 
that
$x_K \in P^\alpha_{F(T,L), X}$ and $F(T,L)(x) = W(L)(x)$.}
\end{itemize}
  From here forward, we 
drop the subscripts on $P^\alpha_{F(T,L), X}$, whenever 
possible, \august{always omitting $X$, and} writing $P^\alpha$ 
\august{instead of $P^\alpha_{F(T,L)}$}.  When subscripts are included, it is 
because we reference the derivation applied to a different function,
typically $F(T_n, J_n\langle K\rangle)$ for some $n$ and some $K$.

In the base case, the tree $T = \{\emptyset\}$ has $|T|_{ls} = 1$, and 
$F(T,L) = W(L)$, 
so it also has rank 1, \august{and the existence of appropriate $x_K$ 
is immediate}.  Let $T$ be given, with $|T|_{ls} = \alpha + 1$. 
\august{If $\alpha = 0$, then just as in the base case, $P^0 = L$, 
so letting $x_K = \min K$ for each $K \in \flat(L)$ suffices.  So 
assume that $\alpha \geq 1$.}
By induction, for each $n$, $|T_n|_{ls} = |F(T_n,J_n\langle K \rangle)|_X$.  
We know that
$$P^\alpha \subseteq \bigcup_{K \in \flat(L)} \{\min K\} \cup \bigcup_n J_n\langle K \rangle,$$ 
because $F(T,L)$ \august{is absolutely continuous} outside this set. 
\august{By the discussion following Definition \mbox{\ref{tildeF}}},
because $|T_n|_{ls} \leq \alpha+1$ 
for all $n$, we have $P^{\alpha+1} \cap J_n\langle K \rangle = \emptyset$ for 
all $n$ \august{and $K$}. 
Similarly, because $|T_n|_{ls} \leq \alpha$ for sufficiently large $n$,
we also have $P^\alpha \cap J_n\langle K \rangle = \emptyset$ for sufficiently 
large $n$.  Therefore, if $\min K \in P^\alpha$, it is isolated, so 
\august{for all $K$}, we have $\min K \not \in P^{\alpha+1}$. 
This shows that $|F(T,L)|_X \leq \alpha +1$.

On the other hand, suppose that $|T|_{ls} = \alpha+1$.  
\august{Fix $K \in \flat(L)$.  Observe 
that finding a point $x_K$ as above suffices to complete the argument, 
because $x_K$ witnesses that $P^\alpha \neq \emptyset$, and so 
$|F(T,L)|_X \geq \alpha+1$.  We consider three 
cases, which are not mutually exclusive, but which exhaust all possibilities.} 

{\bf Case 1.} Suppose $|T_n|_{ls} = \alpha+ 1$ 
for some $n$.  Then, letting $H\in \flat(J_n\langle K \rangle)$
be leftmost, by induction
let $x_K \in H$ with $x_K \in P^\alpha_{F(T_n,J_n\langle K\rangle)}$
and $F(T_n,J_n\langle K\rangle)(x_K) = W(J_n\langle K \rangle)(x_K)=0$, 
where the last equality follows because $H$ was chosen leftmost. 
Observe that 
$x_K \in P^\alpha$ as well, and $F(T,L)(x_K) = W(L)(x_K) + F(T_n,J_n\langle K\rangle)(x_K) = W(L)(x_K).$

{\bf Case 2.} Suppose that $\alpha$ is a limit and $\alpha = \limsup_n |T_n|_{ls}$.  In this case we may let $x_K = \min K$. 
It is immediate that $F(T,L)(x_K) = W(L)(x_K)$, because no other summand 
in Definition \ref{tildeF} has $x_K$ in its support.  For any $\beta < \alpha$,
there are infinitely many $n$ such that $|T_n|_{ls} > \beta$.  
For each such $n$, by induction there is some $x_n \in J_n\langle K \rangle$ 
such that $x_n \in P^\beta_{F(T_n,J_n\langle K \rangle)}$.  We also 
have $x_n \in P^\beta$ for all such $n$.  Since $x_K$ is the limit of 
of this collection of $x_n$, it follows that $x_K \in P^\beta$, 
because $P^\beta$ is closed.  This was done for an arbitrary $\beta< \alpha$, 
so $x_K \in P^\alpha$.

{\bf Case 3.}
Suppose that $\alpha = \beta+1$ and $\alpha = \limsup_n |T_n|_{ls}$. 
Then for infinitely many $n$, we have 
$|T_n|_{ls} = \alpha$.  Let $x_K = \min K$.  As above, we have 
$F(T,L)(x_K) = W(L)(x_K)$, so it remains only to show that $x_K \in P^\alpha$.

For each $n$ such that $|T_n|_{ls} = \alpha$, and each 
$H \in \flat(J_n\langle K\rangle)$, by induction let $x_{n,H} \in H$ 
be such that $x_{n,H} \in P^\beta_{F(T_n, J_n\langle K \rangle)}$ and
\begin{equation}\label{inductive-equality}F(T_n, J_n\langle K \rangle)(x_{n,H}) = W(J_n\langle K\rangle)(x_{n,H}).\end{equation}
We also have each such $x_{n,H} \in P^\beta$.  Now observe that 
the variation of $W(J_n\langle K \rangle)$ on 
$P^\beta$ is at least 1, because $P^\beta$ 
contains one point $x_{n,H}$ in each interval 
$H \in \flat(J_n\langle K\rangle)$. By (\ref{inductive-equality}), 
the variation of $F(T_n, J_n\langle K \rangle)$ is at least 1 on $P^\beta$, 
so in view of the discussion following 
Definition \ref{tildeF}, the variation of $F(T,L)$ 
is also at least 1 on $P^\beta \cap J_n\langle K \rangle$.  
In any neighborhood of $x_K$, there are infinitely many intervals 
$J_n \langle K \rangle$ for which this is true.  Therefore,
 $F(T,L)$ is not $VB$ on 
$P^\beta$ in any neighborhood of $x_K$.  
This shows that $x_K \in P^\alpha$.\end{proof}

\begin{cor} Let $Y \in 2^\omega$ and $\alpha < \omega_1^Y$. 
For any $\Sigma^0_{2\alpha}(Y)$ set $B$, there is a $Y$-computable 
reduction from $(B, \neg B)$ to $(ACG_{\ast\alpha}, \neg VBG_\alpha)$.
\end{cor}

\begin{theorem}
The sets $VBG, VBG_\ast, ACG, ACG_\ast$ are all $\Pi^1_1$-complete.
\end{theorem}

This answers a question of Walsh \cite{walsh}, 
who asked whether $ACG_\ast$ was $\Pi^1_1$-complete.  Walsh 
also showed that the graph of Denjoy integration is a $\Pi^1_1$, 
non-Borel subset of $M(I)\times C(I)$.
He asked whether that set is $\Pi^1_1$-complete, 
a question which we answer in the affirmative.
From here on, 
let $F_T = F(T,I)$.

\begin{lemma}\label{lem:19}
For all trees $T\subseteq \omega^{<\omega}$, $F_T$ is a.e differentiable, and the map 
$T \mapsto (F_T', F_T)\in M(I)\times C(I)$ is computable.
\end{lemma}
\begin{proof}
For each $\ell$, let $$G_\ell = \sum_{\substack{\sigma \in T : |\sigma| < \ell \\ \text{ and } \max(\sigma) < \ell \\ H \in \mathcal C^I_\sigma}} W(H).$$
Then $\lim_{\ell\rightarrow \infty} G_\ell = F_T$.  Observe that 
$G'_\ell$ is a.e. equivalent to an ideal point of $M(I)$. 
We claim that in $M(I)$, $\lim_{\ell\rightarrow \infty} G'_\ell = F_T'$.  For this 
it suffices to observe that $G_\ell = F_T$ on any interval where $G_\ell' \neq 0$, 
and also on any interval that is disjoint from all intervals of
$\mathcal C^I_\sigma$ for $\sigma \in T$ with $|\sigma| \geq \ell$ or 
$\max(\sigma) \geq \ell$.  Regardless of whether $T$ is well-founded,
the measure of these intervals of agreement 
approaches 
1 in the limit, and the convergence is effective.
\end{proof}

\begin{theorem} The set $\{(f,F) \in M(I)\times C(I) : F \in ACG_\ast \text{ and } F'=f \text{ a.e.}\}$ 
is $\Pi^1_1$-complete.
\end{theorem}
\begin{proof}
The map 
$T \mapsto (F_T', F_T)$ provides a computable reduction from $WF$ to 
$\{(f,F) \in M(I)\times C(I) : F\in ACG_\ast \text{ and } F'=f \text{ a.e.}\}$.
\end{proof}

The next and final result of this section
concerns $\{f \in M(I) : f \text{ is 
Denjoy integrable}\}$.  Walsh showed that this set is $\mathbf 
\Sigma^1_2$ and not $\mathbf \Sigma^1_1$, and asked for better bounds, 
which we give in Theorem \ref{thm:MI}.  We need a lemma.

\begin{lemma}\label{lem:D_integrable}
If $T \not\in WF$, then $F_T'$ is not Denjoy integrable. 
\end{lemma}
\begin{proof}
Let
$P = \cap_{\ell \in \omega} \overline{ \cup D_\ell},$ where
$$D_\ell 
= \{ H : H \in \mathcal C^I_\sigma, \sigma \in T, |\sigma| = \ell 
\text{ and } T_{\sigma} \not\in WF\}$$
Then $F_T$ is not $VB$ on $P \cap U$ for any open interval $U$ 
such that $P\cap U \neq \emptyset$, by a modification of 
the argument in Proposition \ref{prop:14}.  Let $H_0 \subseteq U$ 
with $H_0 \in D_M$ for some $M$, and let $H\subseteq H_0$ 
with $H\in D_{M+N}$.  There are at least $2^N$ such $H$, all disjoint. 
Then as before, the variation of $F_T$ on each such $H$ is 
at least 1. Furthermore, if $\flat(H) = \{K_0,K_1,\dots, K_r\}$,
listed from left to right, then this variation is witnessed by 
the intervals $(\max (K_i\cap P), \min (K_{i+1}\cap P))$ 
for each $i<r$, and these intervals are disjoint from $P$.

We claim that 
the restriction 
of $F_T$ to any closed interval $J$ disjoint from $P$
is $ACG_\ast$.  Such an interval is disjoint from $\cup D_\ell$ 
for some $\ell$, and therefore $F_T\uhr J = F_S \uhr J$, 
where $S$ is the well-founded tree obtained from $T$ by 
removing all $\sigma$ \august{of length at least $\ell$}
such that $T_{\sigma \uhr \ell} \not\in WF$.
Because $F_S$ is $ACG_\ast$, $F_T \uhr J$ is $ACG_\ast$.

Now suppose, for the sake of contradiction, that there is $G \in ACG_\ast$ 
with $G' = F_T'$.  Then $G$ would be $VB$ on some portion of $P$. 
But for each connected component $(c,d)$ of $I \setminus P$, 
we have $G(d) - G(c) = F_T(d) - F_T(c)$ (because, using the continuity 
 of $F_T$ and $G$, both are equal to the limit as $J$ expands to 
$(c,d)$ of the Denjoy integral of $F_T'\uhr J$), and $F_T$ is not $VB$ on 
any portion of $P$, via a sequence that uses such intervals $(c,d)$ 
as witnesses.  Such a sequence also witnesses that $G$ has 
variation at least $2^N$ on $P \cap U$.
\end{proof}

\august{We also need a slight variation on the result, originally 
due to Ajtai, that for
sequences $\overline f \in C(I)^\omega$, 
if $\overline f$ converges to $f$ and $f$ is the derivative 
of an everywhere differentiable function $F$, then 
$F \in \Delta^1_1(\overline f)$. } 

\begin{theorem}[essentially Ajtai, unpublished]\label{ajtai}
\august{If $f \in M(I)$ and $F \in ACG_\ast$ with $F' = f$ a.e., then $F \in \Delta^1_1(f)$.}
\end{theorem}
\begin{proof}
\august{We refer the reader to Dougherty and Kechris }\cite[pg. 162]{dk} 
\august{for 
the proof of Ajtai's theorem; we only need to check that the proof 
still works if the function $f$ is 
given as a  
measurable function (i.e as a Cauchy sequence in the 
metric space $M(I)$) that is Denjoy integrable, 
rather than as an element of $C(I)^\omega$
which must be 
the derivative of an everywhere differentiable function.

For the first modification, it suffices to verify that 
it is uniformly $\Delta^1_1$
to determine whether a given function $f \in M(I)$ is
Lebesgue integrable, and if so, what is $\int f$.  
Recall that a 
non-negative function 
$f\in M(I)$ is Lebesgue integrable if and only if the sequence 
$\langle f_n \rangle_{n\in \omega}$ is Cauchy in the $||\cdot||_1$ 
norm, where $f_n(x) = \min(n,f(x))$.  The operation $||\cdot||_1$ 
is computable on bounded measurable functions when the bound is 
also given as an input.  Therefore, it is uniformly $\Delta^1_1$ 
to determine whether an arbitrary given 
$f$ is integrable, by applying 
the above criterion to $|f|$.  If $f$ is integrable, we may calculate
$\int f = \lim_n \int f^n$, where 
$f^n(x) = \max(-n, \min(n, f(x)))$.  
Therefore, all the operations 
needed to carry out the Denjoy totalization process on $f$ 
are uniformly $\Delta^1_1$.

The second difference in the theorem statement is that we only assume 
$F \in ACG_\ast$, rather than $F$ being everywhere differentiable. 
However, the only place it was used that $F$ is everywhere 
differentiable (aside from making $f$ everywhere defined in its  
representation as a Baire 1 function) is the assumption that the 
totalization process does terminate.  }
\end{proof}

\begin{theorem}\label{thm:MI} The set $\{f \in M(I) : f \text{ is Denjoy integrable }\}$ 
is $\Pi^1_1$-complete.
\end{theorem}
\begin{proof}
By Theorem \ref{ajtai}, for $f \in M(I)$,
$$f \text{ is Denjoy integrable} \iff \exists F \in \Delta^1_1(f)[{\mathaugust F \in ACG_\ast \text{ and }} F'=f \text{ a.e.}]$$
which is $\Pi^1_1$.

Now consider the completeness direction.
By Lemma \ref{lem:19}, the map $T\mapsto F_T'$ is computable, 
and $T \in WF$ if and only if $F_T'$ is Denjoy integrable 
by Lemma \ref{lem:D_integrable}.
\end{proof}

\section{Descriptive results on the Denjoy hierarchies}\label{sec:descriptions}

\subsection{Descriptive complexity of $VBG$ and $VBG_\ast$ hierarchies}
Our goal now is to give the precise descriptive complexity of 
the sets $VBG_{\ast\alpha}, VBG_\alpha$ and $ACG_{\ast\alpha}$. 
For $ACG_\alpha$, we will only give an upper bound on the descriptive
complexity.

\begin{prop}\label{VBGsigma2alpha}
If $\alpha < \omega_1^Y$, then $VBG_{\ast\alpha}$ and $VBG_\alpha$ are
$\Sigma_{2\alpha}(Y)$.
\end{prop}

\begin{proof}
By effective transfinite recursion.  Let $X$ be one of $VB, VB_\ast$, 
and let $p,q \in \mathbb Q$.
When $\alpha = 1$, we have $P^1_{F,X} \cap (p,q) = \emptyset$ if and only 
if $F$ has bounded variation on the 
interval $[p,q]$.  This property  is $\Sigma^0_2$
 uniformly in $F$, $p$ and $q$.

Supposing that $\{(F,p,q) : [p,q]  \cap P^\alpha_{F,X} = \emptyset\}$ is 
$\Sigma^0_{2\alpha}(Y)$ uniformly in $\alpha$, let us show 
$\{ (F, p,q): [p,q] \cap P^{\alpha+1}_{F,X} = \emptyset\}$ is 
$\Sigma_{2\alpha+2}(Y)$ uniformly in $\alpha$.
We have by Proposition \ref{prop:linearized},
\begin{align*} P^{\alpha+1}_{F,VB} \cap [p,q] = \emptyset &\iff 
\exists [p',q']\supsetneq [p,q] \text{ such that } F_{P^\alpha} 
\text{ is $VB$ on } [p',q'] \\
P^{\alpha+1}_{F,VB_\ast} \cap [p,q] = \emptyset &\iff 
\exists [p',q'] \supsetneq [p,q] \text{ such that } F_{E,\ast} 
\text{ is $VB$ on } [p',q'], \text{ where } E= P^\alpha.
\end{align*}

So it is sufficient to show that 
there is a uniform procedure for computing, given $\eps$, approximations
to $F_{P^\alpha}$ and $F_{P^\alpha,\ast}$ that are correct to within $\eps$, 
using an oracle which can answer $\Sigma^0_{2\alpha}(F,Y)$ questions,
for example questions of the form
``$P^{\alpha}_{F,X} \cap [a,b] = \emptyset$?''

Using $F$ and the compactness of $I$, compute a number $N$ large 
enough that $$\omega(F,[k/2^N, (k+1)/2^N]) < \eps$$ for all $k<2^N$.
For each $\ell<k\leq 2^N$, ask whether 
$P^{\alpha}_{F,X} \cap [\ell/2^N, k/2^N] = \emptyset$.  This identifies 
(up to an error of $1/2^N$ on each side) the connected components 
of $I\setminus P^{\alpha}_{F,X}$.  
Let $E$ be the corresponding approximation to $P^{\alpha}_{F,X}$; 
that is, 
$$E = I \setminus \cup \{(\ell/2^N, k/2^N) : 
P^{\alpha}_{F,X} \cap [\ell/2^N, k/2^N] = \emptyset\}.$$
Let $F^\eps$ be an approximation of $F$ 
correct to within $\eps$. Then return $F^\eps_{E}$ or $F^\eps_{E,\ast}$ 
as appropriate.  The returned function will be correct to 
within $\eps$ on $E$; on components of $I\setminus E$ that 
were too small to find, $F^\eps_{E,\ast}$ could be off by at 
most $4\eps$; on components of $I\setminus E$ whose 
endpoints were only approximated, $F^\eps_{E,\ast}$ could be 
off by at most $5\eps$.  The errors for $F^\eps_E$ are less.

Therefore, by induction,
$\{(F,p,q) : P^{\alpha+1}_{F,X} \cap [p,q] = \emptyset\}$ 
is $\Sigma^0_{2\alpha+2}(Y)$, uniformly in $\alpha$.

Finally, if a limit ordinal $\lambda$ is given as a $Y$-effective sequence 
$\alpha_n$ with $\lim_{n\rightarrow\infty} \alpha_n = \lambda$, then  
$[p,q]\cap P^\lambda_{F,X} = \emptyset 
\iff \exists n [p,q]\cap P^{\alpha_n} = \emptyset$.  
Since the statements $[p,q]\cap P^{\alpha_n}_{F,X}$ are 
uniformly $\Sigma^0_{2\alpha_n}(Y)$, we have 
$\{(F,p,q) : [p,q]\cap P^\lambda_{F,X} = \emptyset\}$ is
$\Sigma^0_{\lambda}(Y) = \Sigma^0_{2\lambda}(Y)$, uniformly 
in $\lambda$.
\end{proof}

Observe that in the proof of Proposition \ref{VBGsigma2alpha}, 
the only place it is used that $X$ is $VB$ or $VB_\ast$, 
rather than $AC$ or $AC_\ast$, is in counting the quantifiers 
of bounded variation in the successor step, to conclude that 
two jumps on an oracle for ``$P^\alpha_{F,X} \cap [a,b] = \emptyset$?'' 
questions would suffice to answer $P^{\alpha+1}_{F,X} \cap [a,b] = \emptyset$ 
questions.  If $X$ were $AC$ or $AC_\ast$, this argument gives 
a bound of 4 jumps.  In the case of $AC_\ast$, in the next section 
we will improve the bound to 2 jumps.  

\begin{cor}
If $\alpha<\omega_1^Y$, $ACG_{\ast\alpha}$ and $ACG_\alpha$ are 
$\Sigma^0_{4\alpha}(Y)$.
\end{cor}

\subsection{Descriptive complexity of the $ACG_\ast$ hierarchy}

Next we consider the descriptive complexity of $ACG_\ast$ hierarchy.
Already at the first level there is a difference, because 
$P^1_{F,AC_\ast} = \emptyset$ if and only if $F$ is absolutely continuous, 
a $\Pi^0_3$ property.

\begin{theorem}
The set $\{F\in C(I) : F \text{ is absolutely continuous}\}$ is $\Pi^0_3$-complete.
\end{theorem}
\begin{proof}
Our strategy is to define a computable reduction from $2^\omega$ 
to $C(I,I)$ whose output
approximates a version of the Cantor function (Devil's Staircase function) 
which 
will converge to a Cantor-like function only if the input is an element 
of a given $\Pi^0_3$-complete subset of $2^\omega$.  Let $A\subseteq 2^\omega$ 
be such a set and let $g(n)$ be a computable function such that 
for all $X$, $$X \in A \iff \forall n [W_{g(n)}^X \text{ is finite} ]$$
We now define a function $F:[0,1]\rightarrow [0,1]$, uniformly in $X$, 
such that $F$ is absolutely continuous if and only if 
$\forall n [W_{g(n)}^X \text{ is finite}]$ holds.

Effective in $X$, we define a computable sequence of functions $F_s$ which converge effectively and uniformly to the desired computable function $F$.  Let $F_0(x) = x$.  Each $F_s$ will be piecewise linear, containing some pieces of slope zero separated by pieces of positive slope.    Wherever $F_s$ is piecewise constant, it is equal to the limiting function $F$.

For each $n$ let $I_n=[\frac{1}{n+2},\frac{1}{n+1}]$.  This is the interval in which $W_{g(n)}^X$'s finiteness or lack thereof will be expressed.  
At stage $s+1$, let $F_{s+1}\uhr[\frac{1}{n+2},\frac{1}{n+1}] = 
F_s \uhr[\frac{1}{n+2},\frac{1}{n+1}] $ for all $n\geq s$ and for all 
$n<s$ such that no new element of $W_{g(n)}^X$ has been enumerated at 
stage $s$.  For those $n$ for which a new element is enumerated into 
$W_{g(n)}^X$, define $F_{s+1}\uhr I_n$ as follows.  For each maximal interval $I\subseteq I_n$ on which $F_s$ is constant, let $F_{s+1} \equiv F_s$ on $I$.  For each maximal interval $I$ on which $F_s$ is linear with positive slope, define $F_{s+1}$ on $I$ to satisfy:
\begin{enumerate}
\item $F_{s+1} = F_{s}$ at the endpoints of $I$
\item $F_{s+1}$ is piecewise linear, continuous, and increasing
\item $F_{s+1}$ has slope zero on $\frac{1}{3}$ of the measure of $I$, 
and positive slope everywhere else on $I$.
\item $F_s$ and $F_{s+1}$ differ by no more than $2^{-s}$ at any point.
\end{enumerate}
This can be accomplished by letting $F_{s+1}\uhr I$ resemble a sufficiently fine staircase.  The effect is that $\frac{1}{3}$ of the  measure of $I$ is given to points at which $F'(x) = 0$.  Thus if $F_s'$ was nonzero on a measure $r$ subset of $I_n$, then $F_{s+1}'$ is nonzero on a measure $\frac{2}{3}r$ subset of $I_n$.

This completes the construction. One may check that $F$ is continuous and of bounded variation.

Now suppose that it holds that $\forall n [W_{g(n)}^X \text{ is finite}]$.  Then for each $n$, there will come a stage $s$ for which $F_s\uhr I_n = F\uhr I_n$, and so the final $F$ is piecewise linear on $I_n$ for all $n$.  Then $F$ satisfies the Lusin $(N)$ property because each $I_n$ satisfies it, and there are only countably many $I_n$.  Thus $F$ is absolutely continuous.

On the other had, suppose that $W_{g(n)}^X$ is infinite for some fixed $n$.  Then letting $Z = \cup_s \{x \in I_n: F_s'(x) = 0\}$, we have $\mu(Z) = \mu(I_n)$, but $F(Z)$ is countable, since for each $s$, $\{F_s(x) : F_s'(x) = 0\}$ is finite.  But $F$ is continuous, so $F(I_n)= I_n$, so $F(I_n\setminus Z)$ has measure $\mu(I_n)$, and $F$ fails to satisfy Lusin's $(N)$.
\end{proof}

Recall that $F \in ACG$ (respectively $ACG_\ast$) if and only if 
$F$ is in $VBG$ (respectively $VBG_\ast$) and $F$ satisfies $(N)$.
{However, the set of continuous functions satisfying 
$(N)$ is $\Pi^1_1$-complete in the set of continuous functions 
\cite{HPZZ1998},}
so for the purposes of this analysis, it is easier to use a
related and strictly stronger condition, Banach's condition $(S)$.

\begin{definition}
A function $F:I \rightarrow \mathbb R$ satisfies $(S)$ if for every 
$\eps$ there is a $\delta$ such that for every set $A \subseteq I$ 
of Lebesgue measure less than $\delta$, its image $F(A)$ has 
measure less than $\eps$.
\end{definition}

In case $F$ is continuous, condition $(S)$ has a $\Pi^0_3$ equivalent 
definition similar to the definition of absolute continuity. 
For the rest of this section we adopt the convention that if $b<a$, 
$[a,b]$ will denote the interval $[b,a]$.

\begin{definition}
A function $F:I\rightarrow\mathbb R$ satisfies \emph{interval-$(S)$} if 
for every 
$\eps$ there is a $\delta$ such that for all non-decreasing sequences 
$a_0,b_0,\dots a_k, b_k$, if $\sum_i (b_i - a_i) < \delta$ and 
the intervals $[F(a_i),F(b_i)]$ are disjoint, then 
$\sum_i |F(b_i) - F(a_i)| \leq \eps$.
\end{definition}

\begin{prop}\label{prop:6}
A function $F \in C(I)$ satisfies $(S)$ if and only if it satisfies 
interval-$(S$).\end{prop}
\begin{proof}
If $F$ satisfies $(S)$, then by the continuity of $F$, we have that 
$[F(a_i),F(b_i)] \subseteq F([a_i,b_i])$.  Therefore, $F$ satisfies
interval-$(S)$ with the same witnesses that it uses to satisfy $(S)$.  
On the other hand, suppose $F$ satisfies interval-$(S)$, and let 
$\varepsilon$ be given.  Let $\varepsilon_0<\varepsilon$.  
Let $\delta$ be the witness that $F$ satisfies $(S)$
for $\varepsilon_0$, 
and let $A\subseteq I$ with $\mu(A) < \delta$.  Then $A$ can be covered 
by $\cup_{i<\omega} (a_i,b_i)$ where $\mu(\cup_{i<\omega} (a_i,b_i)) <\delta$ 
as well.  It suffices to show that for each $k$ and for each $\varepsilon'>0$, 
that $\mu(B_k) < \varepsilon_0 + \varepsilon',$
where $B_k = \cup_{i<k} F([a_i,b_i])$.  By continuity of $F$, $B_k$ is a finite 
union of intervals.  Let $(c_j,d_j)_{j<\ell}$ be chosen so that
\begin{itemize}
\item For each $j$, there is an $i$ such that $[c_j,d_j] \subseteq F([a_i,b_i])$,
\item The $[c_j,d_j]$ are disjoint, and
\item $\mu(B_k \setminus \cup_{j<\ell} [c_j,d_j]) < \varepsilon'$.
\end{itemize}
For each $j$, choose $a_j'$ and $b_j'$ so that for some $i$, $a_j',b_j' \in [a_i,b_i]$,
$F([a_j',b_j']) = [c_j,d_j]$, and
$F(\{a_j',b_j'\}) = \{c_j, d_j\}$.
Then the $(a_j',b_j')_{i<\ell}$ are set up as in the definition of interval-$(S)$,
so $\mu(\cup_{j<\ell} [c_j,d_j])\leq \eps$, so $\mu(B_k) < \eps + \eps'$.  
\end{proof}

\begin{prop}\label{banachS}
If $F$ is $ACG_\ast$, then $F$ fulfills condition $(S)$.
\end{prop}
\begin{proof}
We make use of Banach's property $(T_1)$.
By definition, a function $F$ satisfies Banach's property $(T_1)$ if 
$\{x \in \mathbb R : F^{-1}(\{x\}) \text{ is infinite}\}$ is null.  
It is known (cf. \cite[Thm IX.8.4, pg 284; Thm IX.6.3, pg 279]{saks}) 
that a continuous function satisfies $(S)$ if and only if it satisfies 
$(N)$ and $(T_1)$, and that any $F \in VBG_\ast$ satisfies $(T_1)$. 
Since any $F \in ACG$ satisfies $(N)$, it follows that any $F \in ACG_\ast$ 
satisfies $(N)$ and $(T_1)$, and therefore $(S)$.
\end{proof}

\begin{prop}
Let $\alpha < \omega_1^Y$ and $\alpha >1$.  Then $ACG_{\ast\alpha}$
is $\Sigma^0_{2\alpha}(Y)$.
\end{prop}
\begin{proof}
Recall that a continuous function $F$ is absolutely continuous if and only if 
it is of bounded variation and satisfies property $(N)$.  Furthermore, 
for any closed set $E$, the functions $F_E$ and $F_{E,\ast}$ have 
property $(N)$ if $F$ does, because their linear portions cannot contribute 
to a failure of $(N)$.  Therefore, if $F$ has the property $(N)$, 
then we have not only that $F \in ACG$ (resp. $ACG_\ast$) if and only 
if $F \in VBG$ (resp. $VBG_\ast$), but also that $|F|_{AC} = |F|_{VB}$ 
(resp. $|F|_{AC_\ast} = |F|_{VB_\ast}$) in this case.  This is because 
in the derivation process, the functions $F_E$ (respectively $F_{E,\ast}$) 
are absolutely continuous on an interval 
if and only if they are of bounded variation on that interval.  Since 
$(N)$ is also necessary for members of $ACG$ and $ACG_\ast$, we have
$$ACG_\alpha = VBG_\alpha \cap \{F \in C(I) : F \text{ satisfies } (N)\}$$
and similarly for $ACG_{\ast\alpha}$.  
However, {due to the $\Pi^1_1$-complete complexity \cite{HPZZ1998}
of determining whether $F$ satisfies $(N)$}, we cannot use this.  
What we can use in the case of $ACG_\ast$ 
is the $\Pi^0_3$ property interval-$(S)$, which is implied by $ACG_\ast$ 
and implies $(N)$.  Therefore,
$$ACG_{\ast\alpha} = VBG_{\ast\alpha} \cap \{F \in C(I) : F \text{ satisfies interval-}(S)\}$$
and a similar statement does NOT hold for $ACG$, as there are functions 
in $ACG$ that do not satisfy $(S)$.  Since satisfying interval-$(S)$ 
is only $\Pi^0_3$, this equality and Proposition \ref{VBGsigma2alpha} 
together complete the proof for all $\alpha>1$.
\end{proof}
We conclude this section by summarizing the results into the main theorem 
mentioned in the introduction.
\begin{theorem}
Let $Y \in 2^\omega$, 
let $1<\alpha < \omega_1^Y$, 
and let $A_\alpha = VBG_{\ast\alpha}, VBG_\alpha$ or $ACG_{\ast\alpha}.$  Then
$A_\alpha$ is $\Sigma^0_{2\alpha}(Y)$, and for any 
$\Sigma^0_{2\alpha}(Y)$ set $B$, there is a $Y$-computable reduction from 
$B$ to $A_\alpha$.
In particular, $A_\alpha$ is $\mathbf \Sigma^0_{2\alpha}$-complete,
and if $\alpha< \omega_1^{CK}$, then $A_\alpha$ is $\Sigma^0_{2\alpha}$-complete.
\end{theorem}

\section{Questions}\label{sec:questions}

Letting $|T|$ denote the usual well-founded rank of a tree $T\in WF$, a proof of the 
following in slightly different language can be found in \cite[Propositions 2.12 \& 2.15]
{GreenbergMontalbanSlaman2013}.
\begin{theorem}[\cite{GreenbergMontalbanSlaman2013}]
If $\alpha< \omega_1^{ck}$, then $\{T : |T| \leq \omega\alpha\}$ is 
$\Sigma^0_{2\alpha}$-complete.
\end{theorem}
Therefore, the complexity of initial segments of the usual well-founded rank increases 
by two jumps for every increase of $\omega$ in the rank.  
By contrast, all of 
the natural ranks considered in this paper have an increased complexity of two jumps for every 
increase of 1 in the rank.  We wonder if there are examples of natural $\Pi^1_1$ ranks that 
don't fit one of these two patterns.
\begin{question}
Is there an example of a $\Pi^1_1$ set $A$, 
a natural $\Pi^1_1$ rank which decomposes it as 
$A = \cup_{\alpha<\omega_1} A_\alpha$, 
and ordinals $\beta_\alpha$ such that
each $A_\alpha$ is $\Gamma_{\beta_\alpha}$-complete (here $\Gamma=\Sigma$ or $\Pi$)
such that one of the following holds:
\begin{enumerate}
\item $\beta_{\alpha} + i = \beta_{\alpha+\delta}$ for some $\delta > \omega$ and $i\leq 2$, or
\item $\beta_{\alpha} + i = \beta_{\alpha+1} $ for some $i >2$.
\end{enumerate}
\end{question}
It is open whether $ACG$ provides an example satisfying part (2) of the above. 
Naively, the derivation process producing the classical rank on $ACG$ could require 
$i=4$.
\begin{question}
What are the exact descriptive complexities of the sets $ACG_\alpha$?
\end{question}

Finally, although $\{T \in WF : |T|_{ls} \leq \alpha\}$ and 
$\{T \in WF : |T| \leq \omega\alpha\}$ are both $\Sigma^0_{2\alpha}$-complete,
and thus there is a reduction from one to the other that passes through 
the universal set $\{X : n_0 \in H_{2^a}^X\}$ for appropriate $a$, 
we are not aware of a natural reduction between them.

\begin{question}
Give natural computable functions $f$ and $g$ such that $f^{-1}(\{T : |T|_{ls} \leq \alpha\}) = \{T : |T| \leq \omega\alpha\}$ and $g^{-1}(\{T : |T|\leq \omega\alpha\}) = \{T : |T|_{ls} \leq \alpha\}$.
\end{question}

\bibliographystyle{alpha}
\bibliography{denjoy}

\end{document}